\documentclass[12pt]{amsart}
\pdfoutput=1
\usepackage{amssymb}
\usepackage{amsmath}
\usepackage{amsfonts}
\usepackage[usenames]{color}
\usepackage{graphicx}
\usepackage{array}
\usepackage{psfrag}
\usepackage{color}
\usepackage{ulem}

\definecolor{maroon}{RGB}{250, 0, 150}
\definecolor{orange}{RGB}{255, 80, 0}

\makeatletter
 
 \@addtoreset{equation}{section}
\makeatother

\newtheorem{theorem}{Theorem}
\newtheorem{lemma}[theorem]{Lemma}
\newtheorem{proposition}[theorem]{Proposition}

\theoremstyle{definition}

\newtheorem{observation}[theorem]{Observation}

\theoremstyle{remark}
\newtheorem{remark}[theorem]{Remark}

\numberwithin{equation}{section}

\newcommand{\R}{\mathbb{R}}

\newcommand{\C}{\mathcal{C}}
\newcommand{\calP}{\mathcal{P}}

\newcommand{\ke}{$K_{3,3}$ }

\newcommand{\dfn}[1]{\textit{#1}}
\newcommand{\cthree}[3]{#1 \oplus #2 \ominus #3}
\newcommand{\mt}{\emptyset}
\newcommand{\disk}{\Delta}
\newcommand{\ko}{K_{3,3,1}}

\newcommand{\bibtitle}[1]{\textit{#1}}

\begin{document}

\title{On the number of links in a linearly embedded $K_{3,3,1}$}

\author{Ramin Naimi}
\address{Occidental College, Los Angeles, CA 90041}
\email{rnaimi@oxy.edu}

\author{Elena Pavelescu}
\address{Occidental College, Los Angeles, CA 90041}
\email{pavelescu@oxy.edu}

\thanks{This research was supported in part by NSF grant DMS-0905300.}

\date{\today}

\maketitle

\begin{abstract}
We show there exists a linear embedding of $K_{3,3,1}$ with $n$ nontrivial 2--component links if and only if $n = 1, 2, 3, 4$, or 5.

\end{abstract}


\section{Introduction}

In the early 1980's,  Conway and Gordon~\cite{CG}, and Sachs~\cite{Sa1, Sa2},  showed that $K_6$, the complete graph on six vertices, is \dfn{intrinsically linked}, i.e., every embedding of it in $\R^3$ (or $S^3$) contains two disjoint cycles that form a nontrivial link.
Sachs~\cite{Sa1, Sa2} 
also showed that six other graphs, including
$K_{3,3,1}$ (the tri-partite graph on 3, 3, 1 vertices) are intrinsically linked.
In other words, the minimum number of nontrivial 2--component links in any embedding of any of these graphs is one.
It is not difficult to see that
every graph can be embedded such that every pair of disjoint cycles forms a nontrivial link,
i.e.,
the attained maximum number of nontrivial 2--component links 
among all embeddings of any graph is  the number of pairs of disjoint cycles in the graph.
Thus, there exist embeddings of $K_6$ and $K_{3,3,1}$ with, respectively, 10 and 9 nontrivial 2--component links.
Fleming and Mellor~\cite{FlMe} found either exact values, or lower- and upper-bounds, for the minimum number of nontrivial links in 
$k$--partite graphs on 8 vertices, and in some larger complete bipartite graphs, embedded in $\R^3$.
However, if one restricts attention to \dfn{linear embeddings} (or \textit{straight-edge embeddings}) of graphs, 
i.e., embeddings of graphs in $\R^3$ in which every edge is a straight line segment,
then these minimum and maximum values change.
Hughes~\cite{Hu}, and, independently, Huh and Jeon~\cite{HuJe}
showed that every linear embedding of $K_6$ contains exactly 1 or 3 nontrivial 2--component links.
Work has also been done on the number of nontrivial knots, and links with more than two components, 
in linearly embedded graphs.
Ramirez Alfonsin~\cite{RA} showed that every linearly embedded $K_7$ contains a trefoil knot.
Huh~\cite{Hu2} showed that every linearly embedded $K_7$ 
contains at most three figure eight knots.
Naimi and Pavelescu~\cite{NaPa} showed that every linearly embedded $K_9$ contains
a nonsplit 3--component link.
None of these results hold if one does not require the embeddings to be linear.
The main result of this paper is:

\begin{theorem}
\label{maintheorem}
There exists a linear embedding of $K_{3,3,1}$ with $n$   nontrivial 2--component links if and only if $n = 1, 2, 3, 4$, or 5.
\end{theorem}

We also show that in a linearly embedded $K_{3,3,1}$ with an odd number of nontrivial links all nontrivial links are Hopf links, while
in a linearly embedded $K_{3,3,1}$ with an even number of nontrivial links one nontrivial link is a $(2,4)$--torus link and the rest are Hopf links.

The results of \cite{NaPa} and \cite{RA} were obtained by 
exhaustively checking large numbers of oriented matroids using computer programs. 
We initially proved Theorem~\ref{maintheorem} using oriented matroid theory and a computer program as well. 
The proof we present here uses basic oriented matroid theory; however, it does not rely on any computer program.

We have chosen to study
the graph $\ko$  since it belongs to the Petersen graph family,
i.e., the seven graphs constituting the set of all minor minimal intrinsically linked graphs \cite{RST}.
Furthermore, $K_6$ and $\ko$ together are necessary and sufficient 
for generating the entire Petersen graph family using triangle-Y moves \cite{Sa2}.
If the effect of triangle-Y moves on the number of links in a linearly embedded graph has a ``nice'' characterization,
it, together with our knowing the number of links in $K_6$ and $\ko$, might provide a quick way of finding the number of links in linear embeddings of the remaining Petersen family graphs.

\medskip


We now introduce some notation and terminology that will be used throughout the paper.
Let $S^2 \subset \R^3$ be a sphere centered at a point $O$.
For each point $P \in \R^3$ \dfn{outside} $S^2$, i.e., in the unbounded component of $\R^3 - S^2$,
the \dfn{projection of $P$ onto $S^2$} is 
the point $P'$ where the line segment $PO$ intersects $S^2$.
Given any pair of points $P, Q \in \R^3$ such that the line segment $PQ$ lies entirely outside $S^2$, 
the projection of $PQ$ onto $S^2$ is a geodesic arc $P'Q' \subset S^2$ whose length is less than $\pi$ times the radius of $S^2$.
(To give the reader some perspective: 
we use this setup to project the \ke subgraph of a linearly embedded $K_{3,3,1}$ onto a small sphere centered around the vertex of degree~6 in $K_{3,3,1}$.)

Given two non-antipodal points $V, W \in S^2$,
we denote by  $\mathcal{C}_{VW}$ the great circle determined by $V$ and $W$ on $S^2$.
We say two points $X, Y \in S^2 - \C_{VW}$ are on the \dfn{same side} of $\C_{VW}$ if 
they lie on the same component of $S^2 - \C_{VW}$;
otherwise we say $X$ and $Y$ are on \dfn{different sides} of $\C_{VW}$.

An \dfn{edge} between two non-antipodal points $V, W \in S^2$
is the shortest geodesic arc $VW$ from $V$ to $W$.
Two edges $VW$ and $XY$ in $S^2$ \dfn{cross}
if they intersect in an interior point.
Suppose $V, W, X, Y \in S^2$ are projections of points $\tilde{V}, \tilde{W}, \tilde{X}, \tilde{Y} \in \R^3$
such that $\tilde{V} \tilde{W}$ and $\tilde{X} \tilde{Y}$ are disjoint in $\R^3$,
while $VW$ and $XY$ cross at point $P  =VW \cap XY$ in $S^2$.
If the primage of $P$ on $\tilde{V} \tilde{W}$  is closer to the center of $S^2$ 
than is the preimage of $P$ on $\tilde{X} \tilde{Y}$,
then we say $VW$ is an \dfn{over-strand} and $XY$ is an \dfn{under-strand at $P$},
and we write $VW // XY$.
A graph is \textit{geodesically immersed} in $S^2$ if each of its edges is embedded as a geodesic arc in $S^2$
with length less than half the length of a great circle.
We say a geodesically immersed graph in $S^2$,
together with under- and over-strand information at each crossing, 
is \dfn{realizable} 
if it is the projection of a linearly embedded graph in $\R^3$ that agrees with the given under- and over-strand information at every crossing.

Unless specified otherwise,
the vertices of $K_{3,3,1}$ are assumed to be labeled with $\{1, 2, 3, 4, 5, 6, 7 \}$, 
with the partition $\{1, 3, 5 \} \cup \{2, 4, 6 \} \cup \{7\}$. 
Given a linearly embedded $K_{3,3,1}$ in $\R^3$,
let $S^2$ be a 2--sphere centered at vertex~7, with a sufficiently small radius so that 
the subgraph $K_{3,3} = \ko - 7$  is entirely outside $S^2$.
Then the projection of \ke onto $S^2$
gives a geodesically immersed \ke.
By abuse of notation, we label the vertices of the immersed \ke also with $\{1, 2, 3, 4, 5, 6 \}$, 
with the partition $\{1, 3, 5 \} \cup \{2, 4, 6 \}$.

Throughout the paper,  whenever we say \textit{linking number}, we mean  the absolute value of the linking number.
And we say two disjoint, simple closed curves in $\R^3$ \textit{link} each other if they have non-zero linking number.


\section{Preliminary Results}


In this section we prove a number of lemmas which we will use in the next section to prove our main theorem.

\begin{lemma}
 Let $L$ be a 2--component link consisting of two linearly embedded cycles in $\R^3$ with a total of seven edges. 
 If $L$ has linking number zero, then it is a trivial link.
\label{0linking}
\end{lemma}
\begin{proof}
Since $L$ has seven edges, its components are a triangle, $ABC$, and a quadrilateral, $DEFG$. 
We can assume the seven vertices are in general position.
Let $\disk ABC$ denote the plane region bounded by the triangle $ABC$, 
and $k$ the number of edges of $DEFG$ that intersect $\disk ABC$.
Since $L$ has linking number zero, $k$ must be 0, 2 or 4.
If $k=0$, then $L$ is clearly trivial.

Suppose $k=2$. Let $X$ and $Y$ be the two points where $DEFG$ intersects $\disk ABC$.
If the two edges of $DEFG$ that intersect $\disk ABC$ are adjacent,
then we can assume $X$ and $Y$ lie on the edges $DE$ and $EF$.
So we can isotop $DEFG$ through the disk $\disk EXY$ to make it disjoint from $\disk ABC$;
hence $L$ is trivial.
If the two edges of $DEFG$ that intersect $\disk ABC$  are disjoint, 
then we can assume $X$ and $Y$ lie on the edges $DE$ and $FG$, respectively.
Now, since $L$ has linking number zero, the vertices $E$ and $F$ are on the same side of the plane determined by $A, B$ and $C$.
Hence $\disk XEF$ intersects $\disk ABC$ in only the point $X$.
Also,  $\disk XYF \cap \disk ABC = XY$.
So we can isotop $DEFG$ through the (topological) disk $\disk XEF \cup \disk XYF$
to make it disjoint from $\disk ABC$. Hence $L$ is trivial.

Now suppose $k = 4$.
Let $X, Y, X', Y'$ be the four points where $DE$, $EF$, $DG$, and $FG$, respectively, intersect $\disk ABC$.
Since at most two edges of the quadrilateral $XX'Y'Y$ cross each other, without loss of generality we assume $XX'$ does not intersect $YY'$. 
Then we can isotop $DEFG$ through $\disk DXX'$ and $\disk FYY'$ to make it disjoint from $\disk ABC$.
Hence $L$ is trivial.

\end{proof}

\begin{lemma}
 In every embedding $G$ of $K_{3,3,1}$, the sum of the linking numbers of all 2--component links in $G$ is odd.
\label{oddtotallinking}
\end{lemma}

\begin{proof}
We use an argument similar to Sachs' \cite{Sa1} for $K_6$. There exists an  embedding $G_0$ of $K_{3,3,1}$ that contains exactly one nontrivial link, with linking number 1 
(e.g., see Figure~\ref{12345}(a), where vertex~7 is assumed to be ``high above'' the diagram).
An arbitrary embedding $G$ of $K_{3,3,1}$ can be obtained from $G_0$ by ambient isotopy plus a finite number of crossing changes (i.e., two edges ``passing through'' each other).
A crossing change between two edges changes the linking number of a 2--component link $L$ 
if and only if the two edges are disjoint and each component of $L$ contains one of the two edges;
furthermore, when this is the case, the linking number of $L$ changes by $\pm 1$.
For every pair of disjoint edges in $K_{3,3,1}$, there are exactly two 2--component links $L_1$ and $L_2$ such that each component of each $L_i$ contains one of the two edges.
Hence, with each crossing change, the total linking number changes by $0$, $2$, or $-2$. 
Thus the total linking number in $G$ has the same parity as in $G_0$, i.e., it's odd.
\end{proof}

\begin{observation}
\label{smallarc}
Let  $\mathcal{C}$ represent a great circle on the 2--sphere $S$, 
and denote by $S^+$ and $S^-$ the two hemispheres of $S$ determined by $\mathcal{C}$,
so that $S^+\cap S^-=\C$.
If $X, Y\in S^+$ are two non-antipodal points, then $XY \subset  S^+$.
\end{observation}

\begin{lemma}
A geodesically immersed \ke in $S^2$ has at most nine edge crossings. 
\label{atmost9}
\end{lemma}

\begin{proof}
There are exactly nine quadrilaterals in \ke. We show that in each quadrilateral at most two edges cross each other.
Consider the quadrilateral $1234$ in a geodesically immersed \ke.
Suppose edge $12$ crosses edge $34$.
Then vertices $3$ and $4$ are on different sides of $\C_{12}$. 
See Figure \ref{separation}.
By Observation \ref{smallarc}, edges $14$ and $32$ lie in different hemispheres determined by $\C_{12}$,
and hence they do not cross each other.
\end{proof}

\begin{figure}[htpb!]
\begin{center}
\begin{picture}(160, 136)
\put(0,0){\includegraphics{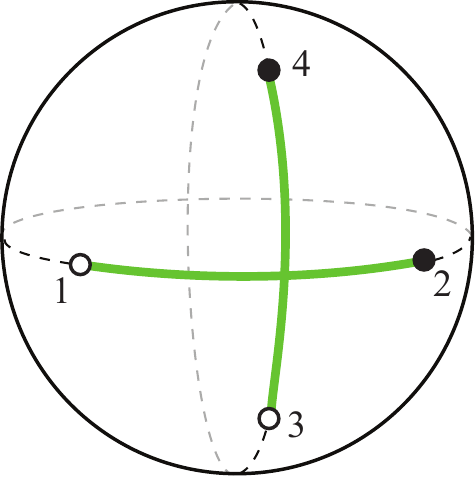}}
\put(15,100){$S_1$}
\put(110,100){$S_2$}
\put(108,30){$S_3$}
\put(16,30){$S_4$}
\end{picture}
\caption{Two intersecting arcs determine four components of the sphere.}
\label{separation}
\end{center}
\end{figure}
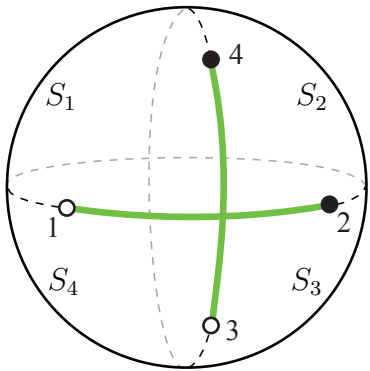


\begin{lemma}
Every geodesically immersed \ke in $S^2$ has at least two edges with zero crossings each.
\label{2edges}
\end{lemma}

\begin{proof}
Suppose, without loss of generality, that edges $12$ and $34$ cross.
Then $S^2 - (\C_{12} \cup \C_{34})$  consists of four components, $S_i$, $i=1,2,3,4$.
See Figure \ref{separation}.
Vertices~$5$ and $6$ cannot lie on $\C_{12}$  or $\C_{34}$ since no four vertices of $K_{3,3,1}$ are coplanar, and hence no three vertices of the immersed \ke lie on the same great circle.
We split the proof into three cases, according to which components $S_i$ contain the vertices~$5$ and $6$. 

\medskip

{Case 1}. Vertices  $5$ and $6$ lie in the same component.
Then we have, up to symmetry, two subcases: $5,6\in S_1$, or $5,6\in S_2$. 
In both subcases, $32$ has no crossings.  
We show that at least one of the other eight edges has no crossings.
If  $5,6\in S_2$, then $14$ has no crossings.
So suppose $5,6 \in S_1$.
If $56\cap 14= \mt$, then $56$ has no crossings, as desired. 
If $56\cap 14\ne \mt$, then vertices~$5$ and $6$ are on different sides of $\C_{14}$. 
Let $S^+$ and $S^-$ denote the two hemispheres determined by $\C_{14}$.
Since 12 and 34 cross,  
vertices $3$ and $2$ lie in the same hemisphere, say $S^+$.
It follows that if $6\in S^-$, then $16$ has no crossings;
and if $5\in S^-$, then $54$ has no crossings.

\medskip

{Case 2}. Vertices~$5$ and $6$ lie in adjacent components.
Then we have, up to symmetry,  two subcases: 
$6\in S_1$ and $5\in S_2$; or $6\in S_1$ and $5\in S_4$. 
In both cases, $32$ has no crossings.  
We show that at least one of the other eight edges has no crossings.
If $6\in S_1$ and $5\in S_2$, then $16$ has no crossings. 
So suppose $6\in S_1$ and $5\in S_4$.
Let $S^+$ and $S^-$ denote the two hemispheres determined by  $\C_{14}$, with $3,2 \in S^+$. 
If $5,6\in S^-$, then $56$ has no crossings.
If $5,6\in S^+$, then $14$ has no crossings.
If $5\in S^+ $ and $6\in S^-$, then $16$ has no crossings.
If $5\in S^-$ and $6\in S^+$, then $54$ has no crossings.

\medskip

{Case 3}. Vertices~$5$ and $6$ lie in opposite components.
Then we have, up to symmetry,  three subcases: 
(3a)~$6\in S_1$ and $5\in S_3$; or  
(3b)~$6\in S_4$ and $5\in S_2$; or 
(3c)~$6\in S_2$, $5\in S_4$. 
In (3a), edges $16$ and $52$ have no crossings.
In (3b), no edge other than $12$ and $34$ has any crossigns.
In (3c), 
if $56 \cap 14 =\mt$ and  $56 \cap 32 = \mt$,  then $14$ and $32$ each have zero crossings.
So, by symmetry, we can assume
 $56 \cap 14 \ne \mt$. 
Then, since vertices 3 and 2 lie on the same side of $\C_{14}$, 
they lie in either the same component or adjacent components of 
$S^2 - (\C_{56} \cup \C_{14})$,
and we are back in Case~1 or 2, respectively.
\end{proof}


\begin{lemma}
Every geodesically immersed  $K_{3,3}$ in $S^2$ has an odd number of crossings.
\label{oddcrossing}
\end{lemma}
\begin{proof}
There is a geodesic immersion $D_0$ of $K_{3,3}$ in $S^2$ with exactly one crossing. See Figure \ref{fig-onecrossing}. 
Let $D_1$ be an arbitrary geodesic immersion of  $K_{3,3}$ in $S^2$.
For $i = 0,1$, let $G_i$ be a linear embedding of $K_{3,3}$ with projection $D_i$.
Let $H: K_{3,3} \times I \to \R^3$ be a linear homotopy (i.e., every point moves in a straight line), taking $G_0$ to $G_1$.
By slightly perturbing $H$,
if necessary, we can assume that for every $t$,
each singularity of the projection of $H(K_{3,3},t)$ onto the sphere is either
(i) a double point at least one whose preimages  is an interior point of an edge (e.g., arising in moves as in Figures~\ref{fig-moves}(a) and \ref{fig-moves}(b)),
 or
 (ii) a triple point all of whose preimages are interiors points of disjoint edges (arising in a Reidemeister~III move as in Figure~\ref{fig-moves}(c)).

Since with each move the parity of the total number of crossings does not change, the desired conclusion follows. 

\begin{figure}[htpb!]
\begin{center}
\begin{picture}(135, 145)
\put(0,0){\includegraphics{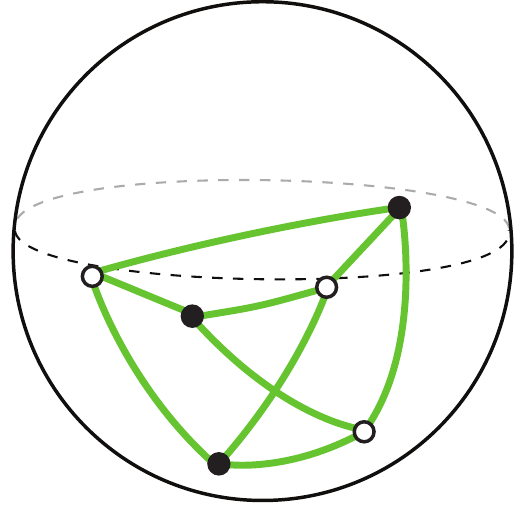}}
\end{picture}
\caption{Geodesically immersed $K_{3,3}$ with one crossing
}
\label{fig-onecrossing}
\end{center}
\end{figure}

\begin{figure}[htpb!]
\begin{center}
\begin{picture}(400, 175)
\put(0,0){\includegraphics{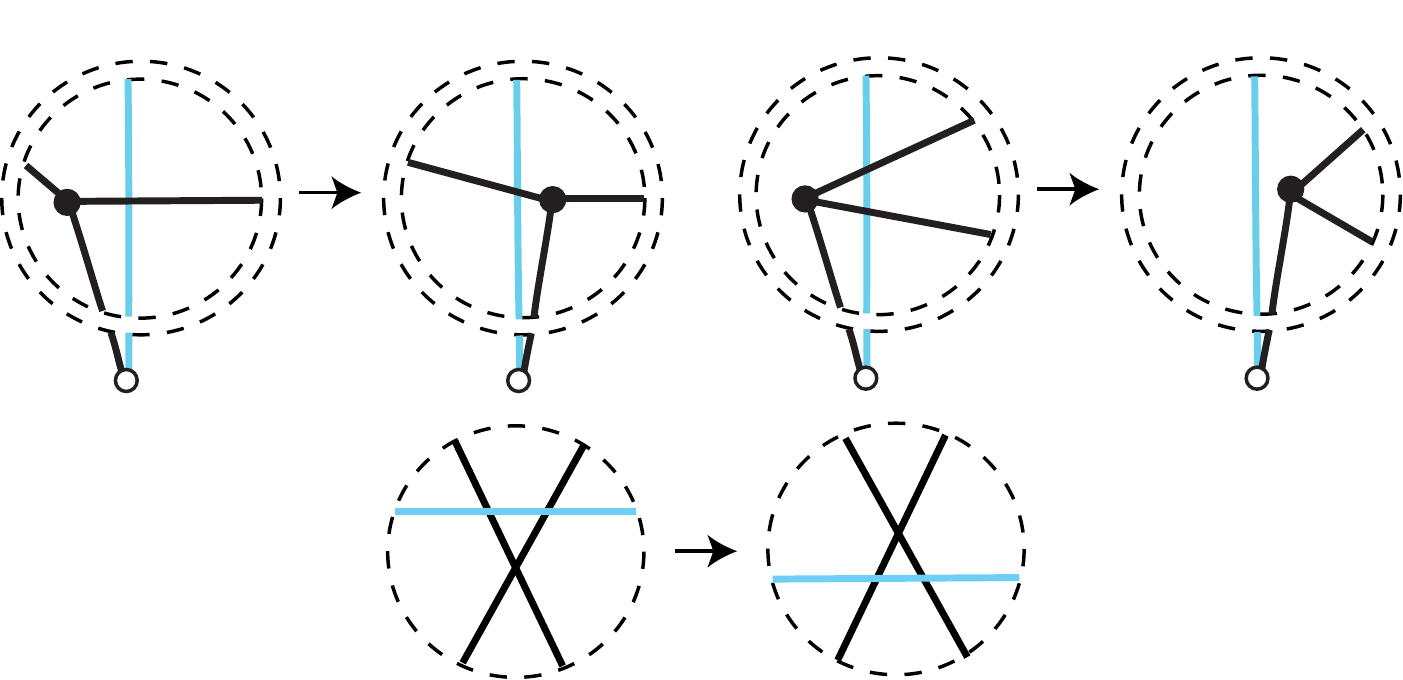}}
\put(90, 80){(a)}
\put(305,80){(b)}
\put(195,-4){(c)}
\end{picture}
\caption{Moves which do not change the parity of the number of crossings}
\label{fig-moves}
\end{center}

\end{figure}
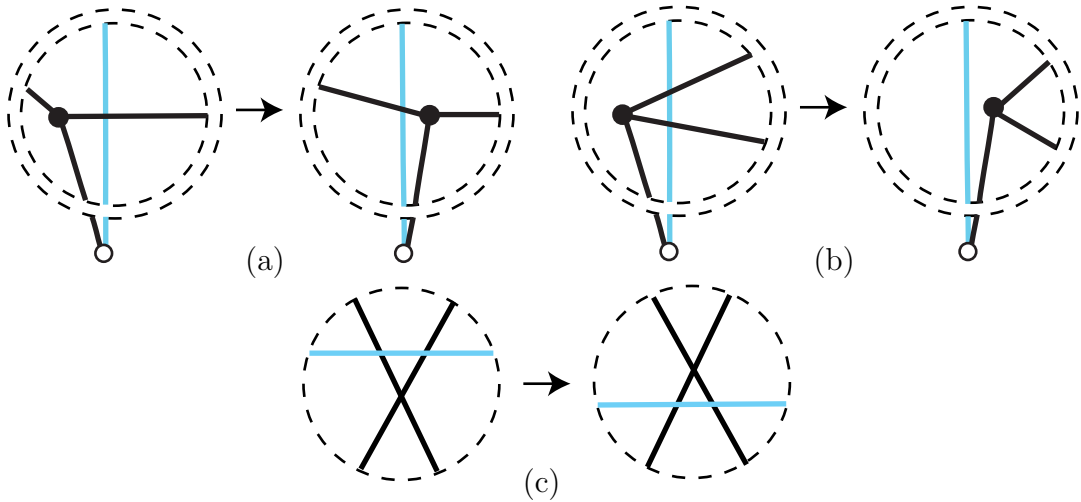

\end{proof}

\begin{remark}
Using the same proof as above, we can in fact show the following:
Suppose in a graph $G$, 
for every every vertex $v$ and every edge $e$ not incident with $v$,
there are an even number of edges disjoint from $e$ and incident with $v$.
Then the number of crossings in a geodesic immersion of $G$ in a 2--sphere
has the same parity as in any other geodesic immersion of $G$ in a 2--sphere.
\end{remark}


\begin{lemma}[Non-realizability Lemma] 
Let $A, B, C, X, Y, Z$ be the vertices of a geodesically immersed graph in a sphere,
such that $AB//XY$, $XY//BC$, $BC//YZ$, and $YZ//AB$.
Then this immersion is not realizable.
\label{NRLemma}
\end{lemma}

\begin{proof}
Suppose toward contradiction that this immersion, $G$, is realizable.
Then $G$ is the projection of
a linearly embedded  graph $\tilde{G}$ with vertices $\tilde{A}, \tilde{B}, \tilde{C}, \tilde{X}, \tilde{Y}, \tilde{Z}$
onto a sphere
such that for each $V \in \{A, B, C, X, Y, Z\}$, 
the projection of $\tilde{V}$ onto the sphere is $V$.

Without loss of generality we may assume that 
the plane determined by the triangle  $\tilde{X}\tilde{Y}\tilde{Z}$ is the plane $z=0$,
and that the sphere is above the $z=0$ plane. 
See Figure \ref{fig-non-realizable}.
Since $AB//XY$ and $YZ//AB$, the edge $\tilde{A}\tilde{B}$ intersects the disk $\disk \tilde{X}\tilde{Y}\tilde{Z}$;
hence the $z$-coordinate of $\tilde{B}$ is positive.
Since $BC//YZ$ and $XY//BC$, the edge $\tilde{B}\tilde{C}$ intersects $\disk \tilde{X}\tilde{Y}\tilde{Z}$, 
and the $z$-coordinate of $\tilde{B}$ is negative, which is a contradiction.
\end{proof}

We will refer to the configuration described in the above lemma as \textit{non-realizable via $ABC$ and $XYZ$}.
Note that to satisfy the hypotheses of the lemma, 
it is sufficient for the crossings to alternate along three of the four edges $AB$, $BC$, $XY$, and $YZ$.

\begin{figure}[htpb!]
\begin{center}
\begin{picture}(160, 120)
\put(0,0){\includegraphics{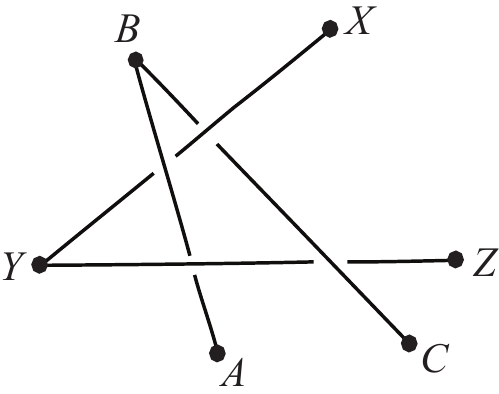}}
\end{picture}
\caption{Non-realizable configuration}
\label{fig-non-realizable}
\end{center}
\end{figure}

\medskip

\begin{lemma}
Let $A, B, C, X, Y, Z$ be the vertices of a geodesically immersed \ke 
that is the projection of the \ke subgraph of a linearly embedded $\ko \subset \R^3$
with vertices  $\tilde{A}$,  $\tilde{B}$, $\tilde{C}$, $\tilde{X}$, $\tilde{Y}$, $\tilde{Z}$, $O$,
onto a sphere centered at $O$,
where $O$ is the vertex of degree six.
Assume each of $AB$ and $BC$ crosses both $XY$ and $YZ$.
Then it is not possible that
three or more of edges  $AB$, $BC$,  $XY$ and $YZ$  will each satisfy one of the following conditions:
\begin{enumerate}
\item[(i)]
the edge has  no other crossings in $K_{3,3}$, 
and the triangle it determines with vertex~$O$ 
links the quadrilateral complementary in $\ko$ to the triangle; or
\item[(ii)]
the triangle determined by the edge and vertex~$O$ has linking number~2 with its complementary quadrilateral.
\end{enumerate}

 \label{CorNR1}
\end{lemma}

\begin{proof}
Suppose an edge $VW$ of \ke crosses two edges on the sphere which are adjacent to each other.
If these two crossings do not alternate along $VW$, 
then their contributions to the linking number of $O \tilde{V} \tilde{W}$ with its complementary quadrilateral
cancel each other out.
Hence, if $VW$ satisfies either condition~(i) or (ii), then these two crossings must alternate.

Now, if  three of the edges  $AB$, $BC$,  $XY$ and $YZ$  each have alternating crossings,
then by  Lemma~\ref{NRLemma}
and the comment following its proof,
we get a contradition.
\end{proof}


\begin{lemma} Let $A, B, P, X, Y, Z$ be the vertices of a geodesically immersed graph in a sphere such that
$YP//AB$, $AB//XY$, and $XY//BP$.  Then this immersion is non-realizable.

\label{CorNR2}
\end{lemma}
\begin{proof}
Suppose this configuration is realizable.
Then we can obtain a realizable configuration by separating the edges $BP$ and $YP$ at $P$ and extending them into new edges $BC$ and $YZ$ such that $BC//YZ$. We do this by raising  $BC$ above $YZ$ a sufficiently small amount so that we have $XY // BC$.
We get a contradiction, since this configuration is non-realizable by Lemma \ref{NRLemma}.

\begin{figure}[ht]
\begin{center}
{\includegraphics{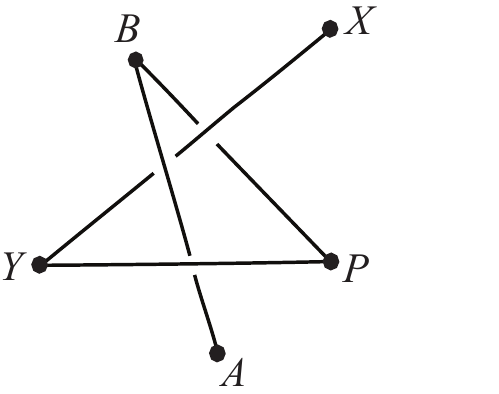}}
\caption{Non-realizable path $ABPYX$}
\label{97}
\end{center}
\end{figure}

\end{proof}

We will refer to the configuration described in the above lemma as \textit{non-realizable via path $ABPYX$}.


\begin{lemma}
\label{lemmaR2}
Let $G$ be a geodesically immersed graph in a closed hemisphere such that no three of its vertices lie on a great circle.
Then there is a graph $G'$  immersed in $\R^2$ with straight edges, and isomorphic to $G$, such that
two edges of $G'$ cross each other if and only if so do the corresponding edges of $G$.
\end{lemma}
\begin{proof}
Since at most two vertices lie on the boundary great circle of the given hemisphere,
by slightly rotating the boundary great circle we can assume that $G$ lies in an open hemisphere $N$. 
Since $\partial N$ is disjoint from $G$,
there is a circle $C \subset N$ (not a great circle), parallel and close to $\partial N$, that is also disjoint from G.
Let $D$ be the flat disk bounded by $C$.
For each vertex $v$ of $G$, the line from $v$ to the sphere's center, $O$, intersects $D$ in a point $v'$.
This gives a projection  $G'$ of $G$ onto $D$.
We see as follows that $G'$ has the same crossings as $G$.
An edge of $G$ with vertices $v$ and $w$ is a subset of the intersection of $N$ with a plane $P_{vw}$ through $O$. 
The intersection of $P_{vw}$ with $D$ contains the edge $v' w'$ of $G'$. 
Now, two edges $vw$ and $xy$ of $G$ cross at a point $vw \cap xy = z$ if and only if
the planes $P_{vw}$ and $P_{xy}$ intersect in the line through $O$ and $z$.
And the latter holds if and only if  $v'w' \cap x'y' = z'$. 
\end{proof}

\begin{lemma}
Let $A, B, X, Y, Z$ be five points in $S^2$, 
no three of which lie on a great circle.
If $AX$ crosses $BY$, and $AZ$ crosses $BX$, then $AZ$ crosses $BY$. 
\label{lemmaM2}
\end{lemma}

\begin{proof}
Since $AX$ crosses $BY$ and $AZ$ crosses $BX$, $B$ and $Z$ must lie on the same side of $\C_{AX}$. 
Let $y\in BY$ be a point on the same side  of $\C_{AX}$ as $Y$, 
sufficiently close to  $AX \cap BY$ so that
the five points $A$, $B$, $X$, $y$, and  $Z$ are contained in a hemisphere. 
Then, by Lemma \ref{lemmaR2} (and by abuse of notation), 
we can assume these five points lie on $\R^2$
 and determine an oriented matroid of rank 3 on five elements.
Since  $AX$ crosses $By$ and $AZ$ crosses $BX$, we have the circuits $(AX, By)$ and $(AZ, BX)$,
where $(\alpha \beta, \gamma \delta)$ denotes the circuit $C=C^+\cup C^-$ with $C^+=\{\alpha, \beta\}$ and $C^-=\{ \gamma, \delta\}$.

Recall Axiom~C3 (weak elimination) of the definition of oriented matroids via circuits (\cite{BLSWZB}, p. 103): 
\begin{quote}
For all  $C_1, C_2\in \mathcal{C}$ with $C_1\ne -C_2$, if $e\in C_1^+\cap C_2^- $, then there exists $C_3\in \mathcal{C}$ such that $C_3^+\subseteq (C_1^+\cup C_2^+) \setminus \{e\}$ and $C_3^-\subseteq (C_1^-\cup C_2^-) \setminus \{e\}$.
\end{quote}

Applying Axiom~C3 to the circuits $(AX, By)$ and $(AZ, BX)$ with $e=X$ gives the circuit $(AZ, By)$, 
which implies $AZ$ crosses $By$, and hence also $BY$.
\end{proof}

\begin{lemma}
\label{4intersections}
Let $X$, $Y$, $A$, $B$, $C$, $D$, $a_1$, $a_2$ be points in $S^2$ such that 
$XY$ crosses $AC$ at $a_1$, and $AD$ at $a_2$,
and $a_1$ is between $X$ and $a_2$.
Suppose also that $XY$ crosses $BD$, and $BX$ crosses $AD$.
Then $BX$ crosses $AC$. 
\end{lemma}

\begin{proof}
The points $A$ and $B$ lie on one side of  $\C_{XY}$, and  $C$ and $D$ on the other.
So $BX \cap AD$ lies on the same side of $\C_{XY}$ as $B$,
which implies $BX \cap AD = BX \cap Aa_2$.
Since $a_1$ is between $X$ and $a_2$,
 $Xa_2$ crosses $AC$.  
Thus, applying Lemma~\ref{lemmaM2} to $X$,  $A$, $B$, $C$, and $a_2$, 
we conclude that $BX$ crosses $AC$. 

\end{proof}

\begin{lemma}
Every linear embedding of $K_{3,3,1}$ has at most one link with linking number 2.
\label{2linking}
\end{lemma}

\begin{proof}
We can assume that no four vertices of the linearly embedded $\ko$ are coplanar.
So this embedding gives an oriented matroid of rank~4 on 7 elements,
where, for any five vertices $v_1, \cdots, v_5$ of $\ko$,
the triangle $v_1 v_2 v_3$ 
(more precisely, the interior of the disk $\disk v_1 v_2 v_3$) 
is pierced by the edge $v_4 v_5$
if and only if $(v_1 v_2 v_3, v_4 v_5)$ is a circuit in the oriented matroid;
and $v_5$ is inside the tetrahedron $v_1 v_2 v_3 v_4$ 
if and only if $(v_1 v_2 v_3 v_4, v_5)$ is a circuit.
In the following, we denote by $\cthree{C_1}{C_2}{e}$ 
the operation of applying Axiom~C3 (see the proof of Lemma~\ref{lemmaM2}) 
to circuits $C_1$ and $C_2$, for edge $e$.

Assume there are two links with linking number 2. 
We have two cases, depending on whether the triangle components of the two links
do or do not share an edge.

\medskip

Case 1: The triangle components of the two links share an edge. 
Say these triangles are 712 and 714.
Since they are components of links with linking number~2,
each of them is pierced by exactly two edges. 
Without loss of generality, we can assume 712 is pierced by edges 34 and 56, so we have the circuits (712, 34) and (712, 56). 
 Since (712, 34) is  a circuit, we cannot have (714, 32) (they have the same underlying set). 
Triangle 714 is therefore pierced by 52 and 36. 
Now, $\cthree{(712, 56)}{(714, 52)}{2}$ yields the circuit (714, 56).
So triangle 714 is pierced by three edges, 56, 52, and 36, which is a contradiction.

\medskip

Case 2: The triangle components of the two links do not share an edge. Say these triangles are 712 and 734.
Up to symmetry, we have two sub-cases.

\medskip

Case 2(a): No edge of one triangle pierces the other triangle; i.e., 34 does not pierce triangle 712 and 12 does not pierce triangle 734. 
As in Case~1, a triangle with linking number 2 is pierced exactly twice.
So 712 is pierced by 36 and 54 only; and 734 by 16 and 52 only.
Thus we have the circuits  (712, 36), (712, 54), (734, 16), and (734, 52),
and no other circuit of the form $(712, ab)$ or $(734, ab)$.

From the second and fourth circuits, $\cthree{(712, 54)}{(734, 52)}{4}$ yields a circuit $C$ with $C^+\subset \{7,1,2,3\}$ and  $C^-\subset \{2, 5\}$.
This circuit can be one of (7123, 5) or (713, 52).
If $C=(7123, 5)$, then $\cthree{(7123, 5)}{(712, 36)}{3}$ yields the circuit (712, 56), which is a contradiction.
So we have $C= (713, 52)$; then $\cthree{(713, 52)}{(712, 36)}{3}$ gives a circuit  $D$ with $D^+\subset \{7,1,2\}$ and  $D^-\subset \{2, 5, 6\}$.
As (712, 56) is ruled out, we must have $D=(71, 256)$. 
On the other hand, $\cthree{(712, 36)}{(734, 16)}{3}$ gives a circuit  $E$ with $E^+\subset \{7,1,2,4\}$ and  $E^-\subset \{1, 6\}$. 
The circuit $E$ can be one of (7124, 6) or (724, 16).
If $E= (7124, 6)$, $\cthree{(7124, 6)}{ (712, 54)}{4}$ gives the circuit  (712, 56) (contradiction).
If $E= (724, 16)$, $\cthree{(724, 16)}{(712, 54)}{4}$ gives a circuit   $F$ with $F^+\subset \{7,1,2\}$ and  $F^-\subset \{1, 5, 6\}$. 
This circuit can be one of (712, 56) or (72, 156). 
But we already have the circuit $D= (71, 256)$ with the same underlying set.
Thus we again have a contradiction.

\medskip

Case 2(b):
One triangle is pierced by one edge of the other triangle. 
So, without loss of generality,
we can assume  triangle 712 is pierced by edge 34.
Recall that a triangle with linking number 2 
is pierced by exactly two disjoint edges.
It follows that
we have the circuits  (712, 34), (712, 56), (734, 16) and (734, 52);
and that triangles 712 and 734 are not pierced by any other edges.

Now,  $\cthree{(712,56)}{(734,16)}{1}$ 
gives a circuit $G$ with $G^+ \subset \{7,2,3,4\}$ and  $G^- \subset \{5,6\}$.
First note that $G^- = \{5\}$ is impossible since ${(7234,5)}$ has the same underlying set as $(734,52)$.
And if $G^- = \{6\}$, then
$\cthree{(7234,6)}{(734,52)}{2}$ gives $(734,56)$,
which is a contradiction.
It follows that $G^- = \{5,6\}$, and hence $G = (abc,56)$ where $\{a,b,c\} \subset \{7,2,3,4\}$.
So $abc = 723, 724, 734$, or $234$.
But $abc \ne 734$ since $(734,56)$ cannot be a circuit.
And if $abc = 723, 724$, or $234$, then $\cthree{(abc,56)}{(734,52)}{2}$ gives $(734,56)$, a contradiction.

\end{proof}

\begin{proposition}
In a linearly embedded $K_{3,3,1}$ with an odd number of nontrivial links,
 all nontrivial links are Hopf links.
In a linearly embedded $K_{3,3,1}$ with an even number of nontrivial links, 
one nontrivial link is a $(2,4)$--torus link and the rest are Hopf links.
\label{linktypes}
\end{proposition}
\begin{proof}
By Lemma \ref{oddtotallinking}, the sum of the linking numbers of all links
in a linearly embedded $K_{3,3,1}$ is odd; and by Lemma~\ref{2linking}, at most one link has linking number~2.
It follows that all the other nontrivial links have linking number~1, 
since in a linearly embedded link consisting of one triangle and one quadrilateral,
a linking number of 3 or higher is not possible.
If there are an odd number of non-trivial links, then all the non-trivial links have linking number~1,
which implies they are Hopf links since they are linearly embedded.
If there are an even number of non-trivial links, then one nontrivial link has linking number~2,
which implies it's a $(2,4)$--torus link since it's linearly embedded,
and the rest are Hopf links.
\end{proof}


\begin{lemma}
A geodesically immersed \ke in $S^2$ with 9 crossings has a unique crossing pattern; i.e., 
up to relabeling its vertices, the 9 crossings come from the following 9 pairs of edges:
$(14, 32)$, $(16, 32)$, $(14, 52)$, $(16, 52)$, 
$(14, 36)$, $(14, 56)$, $(32, 54)$, 
$(32, 56)$, $(36, 54)$;
and each of $12$ and $34$ has no crossings.
\label{9crossings}
\end{lemma}
\begin{proof}
By Lemma~\ref{2edges}, the immersed $K_{3,3}$ has two edges that have no crossings.
These two edges cannot possibly be adjacent,
since if say $12$ and $14$ have no crossings,
then the quadrilateral $1234$  has no self-crossings;
and this is a contradiction since 
\ke has only 9 quadrilaterals and each immersed quadrilateral has at most one self-crossing.

Thus the two edges with no crossings  are disjoint; we can assume they are $12$ and $34$. 
Since every quadrilateral has a self-crossing, we must have the following crossings (the pairs of edges not listed do not cross):

\begin{enumerate}
\item[(1)]
$1234: 14\cap 32 \ne \mt$
\item[(2)]
$1236$: $16\cap 32 \ne \mt$
\item[(3)]
$1254$: $14\cap 52 \ne \mt$
\item[(4)]
$1256$: $16\cap 52\ne \mt$ 
\item[(5)]
$1436$: $14\cap 36  \ne \mt$
\item[(6)]
$1456$: (a) $14\cap 56 \ne \mt$  xor (b) $16\cap 54 \ne \mt$ 
\item[(7)]
$3254$: $32\cap 54 \ne \mt$ 
\item[(8)]
$3256$: (a) $32\cap 56 \ne \mt$  xor (b) $36\cap 52 \ne \mt$ 
\item[(9)]
$3456$: $36\cap 54 \ne \mt$
\end{enumerate}

It remains to show that in each of (6) and (8), the second pair of edges do not cross each other.
The crossings in (1), (3) and (5) imply $5$ and $6$ lie on different sides of $\C_{14}$. 
Thus $16\cap 54 =\mt$, as desired in (6).
\noindent The crossings in (1), (2) and (7) imply $5$ and $6$ lie on different sides of $\C_{32}$. 
Thus $36\cap 52=\mt$, as desired in (8).
\end{proof}

\begin{lemma}
\label{7X-2Adj}
Suppose \ke is geodesically immersed with exactly 7 crossings.
If edges $12$ and $14$ each have no crossings,
then edge $36$ or edge $56$ has no crossings.
\end{lemma}

\begin{proof}
Since neither $12$ nor $14$ has any crossings,
the quadrilaterals $1234$ and $1254$ have no self-crossings.
So each of the other seven quadrilaterals contains exactly one self-crossing,
as listed below:

\begin{enumerate}
\item $1236$: $16\cap 32 \ne \mt$
\item $1256$: $16\cap 52 \ne \mt$
\item $1436$: $16\cap 34 \ne \mt$
\item $1456$: $16\cap 54 \ne \mt$
\item $3254$: $32\cap 54 \ne \mt$ xor $34\cap 52 \ne \mt$
\item $3256$: $32\cap 56 \ne \mt$ xor {$36\cap 52 \ne \mt$} 
\item $3456$: $36\cap 54 \ne \mt$ xor {$34\cap 56 \ne \mt$}
\end{enumerate}

Now, suppose toward contradiction that $36$ and $56$ each have at least one crossing. 
It follows that, in (6) and (7), 
either $32\cap 56 \ne \mt$ and $36\cap 54 \ne \mt$, or $36\cap 52 \ne \mt$ and $34\cap 56 \ne \mt$.
Since vertex labels $2$ and $4$ are symmetric with respect to the given hypotheses, 
we can without loss of generality assume $32\cap 56 \ne \mt$ in (6) and $36\cap 54 \ne \mt$ in (7). 
Now, $32\cap 56 \ne \mt$ implies $3$ and $2$ are on different sides of $\C_{56}$;  
$36\cap 54 \ne \mt$ implies $3$ and $4$ are on the same side of $\C_{56}$;
and
$16\cap 34 \ne \mt$ implies $1$ is  on the same side of $\C_{56}$ as $3$ and $4$.
It follows that $1$ and $2$ are on different sides of $\C_{56}$, hence  16 and 52 cannot cross each other, contradicting (2).

\end{proof}

\section{Proof of the Main Theorem}


We are now ready to prove Theorem~\ref{maintheorem}.

\begin{proof}
First we show that there exist linear embeddings of $K_{3,3,1}$ with 1, 2, 3, 4, and 5 nontrivial 2--component links.
Figure \ref{12345} shows such embeddings of $K_{3,3,1}$, 
where vertex 7 and the edges incident to it are not drawn. 
Vertex 7 is assumed to be toward the viewer, high above the plane of projection depicted in the figures,
so that each edge incident to vertex~7 is almost vertical.
For each diagram we list the triangles that, together with their respective complementary quadrilaterals,
give a nontrivial link:
(a)~$752$;
(b)~$752$, $754$;
(c)~$714$, $736$, $752$;
(d)~$732$, $734$, $752$, $754$;
(e)~$714$, $732$, $734$, $736$, $752$.
Diagrams (a) and (c) are clearly realizable.
Removing vertex~3 from diagram~(b), 6 from (d), and 4 from (e), gives diagrams that are realizable; 
and the edges incident to each of these vertices all contain only under-strands or only over-strands, 
so the removed vertex and edges can be added back while maintaining realizability.


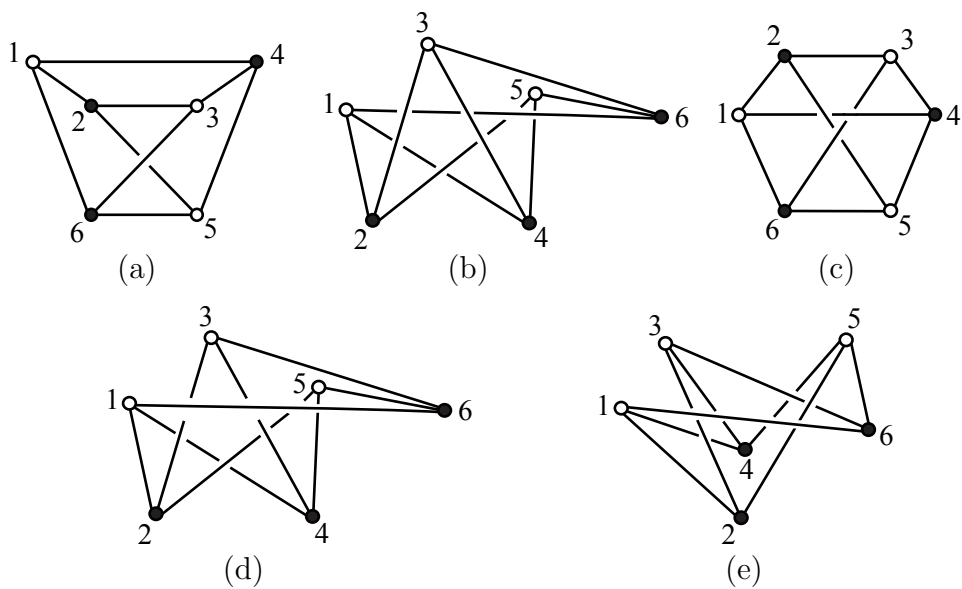
\begin{figure}[ht]
\begin{center}
\begin{picture}(360, 240)
\put(0,0){\includegraphics{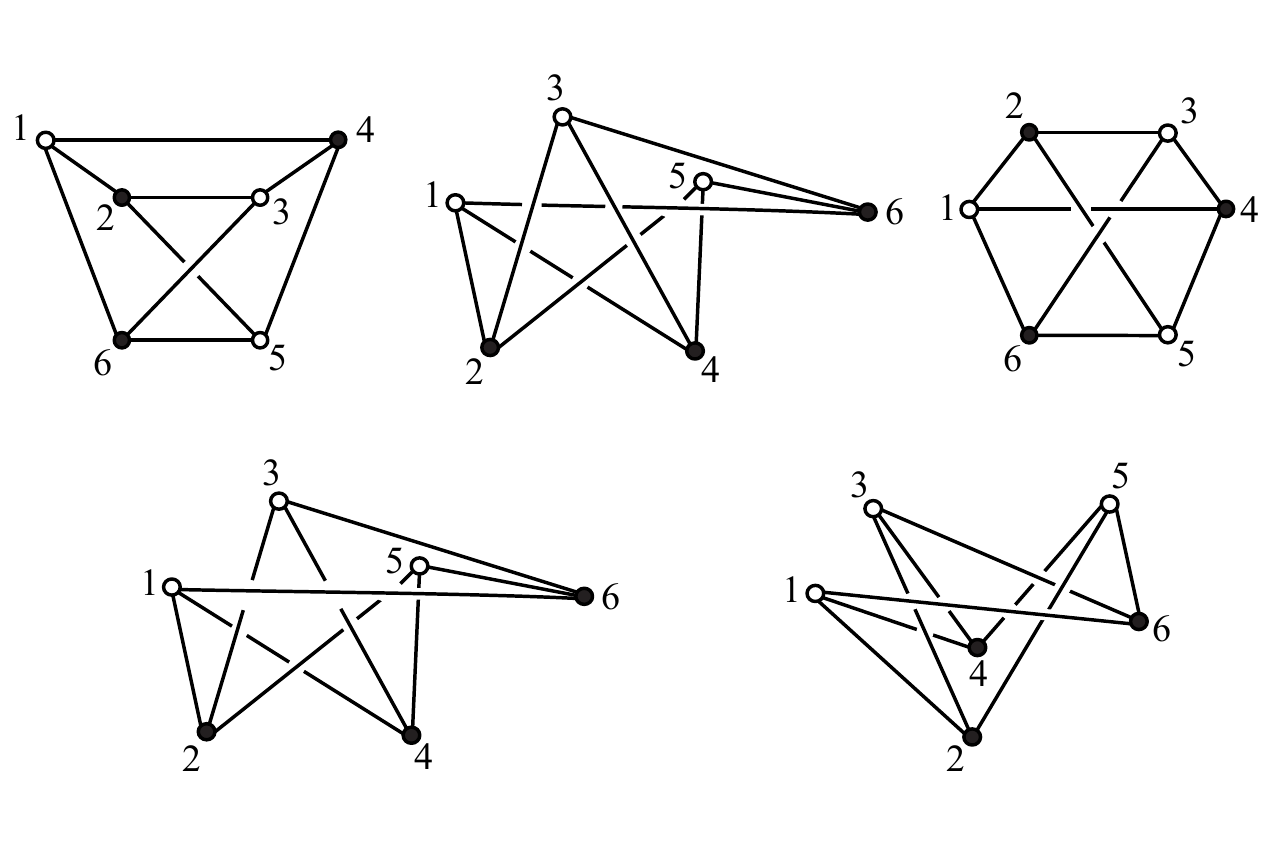}}
\end{picture}
\put(-315,128){(a)}
\put(-190,128){(b)}
\put(-50,128){(c)}
\put(-275,14){(d)}
\put(-85,14){(e)}
\caption{
(a)--(e): $K_{3,3,1}$ with 1, 2, 3, 4, and 5 nontrivial links, respectively.
Vertex~7 of $K_{3,3,1}$ is above the page (not shown).}
\label{12345}
\end{center}
\end{figure}

It remains to show no embedding of $K_{3,3,1}$ can contain more than five nontrivial links.
Given a linear embedding of $\ko$,
let $\calP$ denote the projection of its \ke subgraph
onto a small 2--sphere centered at vertex~7.
Each 2--component link of $K_{3,3,1}$ consists of a triangle containing vertex 7 and its complementary quadrilateral. 
If a triangle is a component of a non-trivial link, its edge disjoint from vertex~7 must have at least one under-strand at a crossing with an edge in the triangle's complementary quadrilateral.
By Lemma \ref{2edges}, at  least two edges of $\calP$ have no crossings.
Therefore, at most seven of the nine triangles in $K_{3,3,1}$ can each be a component of a non-trivial link.
By Lemmas \ref{atmost9} and \ref{oddcrossing}, $\calP$ contains 1, 3, 5, 7 or 9 crossings.
With five or fewer crossings, at most five nontrivial links are possible, as desired.
So, to prove the theorem,  we show that if $\calP$ has 7 or 9 crossings and $K_{3,3,1}$  has 6 or 7 nontrivial links, we reach a contradiction.


\bigskip
\noindent
\textbf{Step I.}   Assume $\calP$ has 9 crossings and $K_{3,3,1}$ has 7 nontrivial links.

\medskip

Up to relabeling, the 9 crossings in $\calP$ are as in Lemma~\ref{9crossings}.
For ease of reference, we list these crossings in the table below,
where each edge in the first row crosses the edges listed under it.

\[
\begin{array}{|c|c|c|c|c|c|c|}
\hline 
14 & 16  & 32 & 36 & 52 & 54 & 56  \\
\hline 
32 & 32 & 14 & 14 & 14 & 32 & 14  \\
36 & 52 & 16 & 54 & 16 & 36 & 32  \\
52 & & 54 & & & & \\
56 & & 56   & &  & & \\
\hline
\end{array}
\]

\bigskip

Since edges $12$ and $34$ have no crossings
and seven of the nine triangles in $K_{3,3,1}$ are in nontrivial links,
each of the seven other edges of $\calP$ must contain at least one under-strand.
Now, if an edge $XY$ crosses two adjacent edges only, 
and the triangle $7XY$ links its complementary quadrilateral , 
then $XY$ contains one under-strand and one over-strand,
i.e., it has alternating crossings.
This is the case for edges 16, 36, 52 and 54.

Since edges 36 and 54 each have alternating crossings, 
edge 32 cannot also have alternating crossings with edges 14 and 54, 
or else we get non-realizability via 145 and 236 by Lemma~\ref{NRLemma}.
So the contributions of 14 and 54 to the linking number of 732 with 1456 add up to zero.
But  then, as this linking number is nonzero, 
32 must have alternating crossings with 16 and 56. 
Since 16 also has alternating crossings with 32 and 52, we get nonrealizability by Lemma~\ref {CorNR2},  via path $32561$.

Therefore a $K_{3,3}$ with 9 crossings cannot be the projection of  a $K_{3,3,1}$ with 7 nontrivial links.


\medskip
\noindent \textbf{Step II.}  Assume $\calP$ has 9 crossings and $K_{3,3,1}$ has exactly 6 nontrivial links.

\smallskip

We assume $12$ and $34$ have no crossings and we have the edge crossings listed in Step~I.
Since we have 6 nontrivial links, and their total linking number is odd,
one of the links must have linking number~2.
A triangle which is a component of a link with linking number~2  is necessarily pierced by two \textit{disjoint} edges of its complementary quadrilateral.
The only candidates for such a triangle  are $714$, $732$ and $756$, since 14, 32 and 56 each cross two disjoint edges. 

First suppose $756$ has linking number~2 with $1234$.
With $14$ oriented from $1$ to $4$, and $32$ from $3$ to $2$, 
these two edges must pierce $756$ with the same sign to give linking number~2.
So $1$ and $3$ must be on one side of $\C_{56}$, $2$ and $4$ the other.
This implies $16$ does not cross $52$, which is a contradiction.

So only $714$ and $732$ remain as candidate triangles.
Note that the 9 crossings listed in the table above remain the same 
if we switch $1$ with $3$ and $2$ with $4$.
Hence, without loss of generality, we can assume that 
$714$ is the triangle that has linking number~2 with its complementary quadrilateral, $3256$.
This can happen only if two opposite edges of $3256$ cross over $14$, and the other two under  $14$
(three over and  one under would give linking number~1 since we'd have two adjacent over-strands).\\

Case~1: Assume $32 // 14$ and $56 // 14$.
Then  $14 // 52$ and $14 // 36$. 
Now, if $52//16$, then  $\calP$ is nonrealizabe via  the path $41652$.
On the other hand, 
if $16//52$, then $752$ does not link,
since $52$ only crosses two adjacent edges.
But then both $736$ and $754$ must  link,
since we have 6 nontrivial links.
This yields $36//54$ and  $54//32$.
Then $\calP$ is non-realizable via $145$ and $236$.

\medskip

Case~2:  Assume $14 // 32 $ and $14 // 56 $.
This implies  $52 // 14$ and $36 // 14$. 
Then all crossings will be reversed as compared with Case~1,
and we get nonrealizability by the same argument.

Hence no triangle can have linking number~2 with its complementary quadrilateral.
And therefore a \ke  with 9 crossings cannot be the projection of a $\ko$ with exactly 6 nontrivial links.


\bigskip

\noindent \textbf{Step III.} Assume $\calP$ has exactly 7 crossings and $K_{3,3,1}$ has 6 or 7 nontrivial links.

\medskip

By Lemma~\ref{2edges}, at least two edges of $\calP$ have no crossings.
First suppose $\calP$ has two adjacent edges that have no crossings; say they are 12 and 14.
Then, by Lemma~\ref {7X-2Adj}, also one of the edges 36 or 56 has no crossings.
Therefore, $\ko$ has only six non-trivial links.
Without loss of generality, assume 36 has no crossings.
Since seven quadrilaterals contain one crossing each, 
we have the following list of pairs of edges that cross each other in each quadrilateral:

\begin{enumerate}
\item $1236$: $16\cap 32 \ne \mt$.
\item $1436$: $16\cap 34 \ne \mt$.
\item $1256$: $16\cap 52 \ne \mt$.
\item $1456$: $16\cap 54 \ne \mt$.
\item $3254$: $32\cap 54 \ne \mt$ xor $34\cap 52 \ne \mt$.
\item $3256$: $32\cap 56 \ne \mt$.
\item $3456$: $34\cap 56 \ne \mt$.
\end{enumerate}

The above list is symmetric with respect to interchanging vertices $2$ and $4$.
So, in line~(5), we can assume $32 \cap 54 \ne \mt$.
Now, since $\ko$ has an even number of nontrivial links,
by Proposition \ref{linktypes},
exactly one of the nontrivial links has linking number~2.
It follows that this configuration is non-realizable by Lemma~\ref{CorNR1}, via 165 and 234: 
edges 34 and 56 satisfy condition~(i) of the lemma, 
and one of edges 16 and 32 satisfies condition~(ii).

So we can assume that $\calP$ has no two adjacent, crossing-less edges;
hence, without loss of generality, edges $12$ and $34$ are crossing-less.
Then the seven crossings are among the pairs of edges in the following set,
which we'll refer to as List~A: 
\begin{enumerate}
\item[(1)]
$1234$: $ 14\cap 32$
\item[(2)]
$1236$: $16\cap 32$
\item[(3)]
$1254$: $14\cap 52$
\item[(4)]
$1256$: $16\cap 52$ 
\item[(5)]
$1436$: $14\cap 36 $
\item[(6)]
$1456$: (a) $14\cap 56 $  xor (b) $16\cap 54 $ 
\item[(7)]
$3254$: $32\cap 54 $ 
\item[(8)]
$3256$: (a) $32\cap 56 $  xor (b) $36\cap 52 $ 
\item[(9)]
$3456$: $36\cap 54 $
\end{enumerate}

We will write $(i)=\mt$ or $(i)\ne\mt$ according to whether or not the pair of edges in line~$(i)$ of List~A cross each other.
For $i = 6$ or $8$, $(i) = \mt$ means both (a) and (b) in line~$(i)$ are $\mt$.
By repeatedly applying Lemma \ref{lemmaM2}, 
we have the following set of implications, which we'll refer to as List~B:

\begin{itemize}
\item  $(1)=\mt$ implies $(2)=\mt$ or $(5)=\mt$.
\item  $(1)=\mt$ implies $(3)=\mt$ or $(7)=\mt$.
\item $(2)=\mt$ implies $(4)=\mt$ or $(8\mathrm{a})=\mt$.
\item  $(3)=\mt$ implies $(4)=\mt$ or $(6\mathrm{a})=\mt$.
\item $(4)=\mt$ implies $(2)=\mt$ or $(8\mathrm{b})=\mt$.
\item $(4)=\mt$ implies $(3)=\mt$ or $(6\mathrm{b})=\mt$.
\item  $(5)=\mt$ implies $(9)=\mt$ or $(6\mathrm{a})=\mt$.
\item  $(7)=\mt$ implies $(9)=\mt$ or $(8\mathrm{a})=\mt$.
\item $(9)=\mt$ implies $(5)=\mt$ or $(6\mathrm{b})=\mt$.
\item $(9)=\mt$ implies $(7)=\mt$ or $(8\mathrm{b})=\mt$.
\end{itemize}

We now divide the rest of Step~III into two parts, according to whether edge $56$ has crossings or not.

\medskip

\textbf{Part~1}. Suppose edge $56$ has a crossing.
Then
$(6\mathrm{a}) \ne \mt$ or $(8\mathrm{a}) \ne \mt$.
Note that List A is symmetric with respect to interchanging vertex 1 with 3 and 2 with 4.
Thus, we can without loss of generality assume $(6\mathrm{a}) \ne \mt$, and hence $(6\mathrm{b}) = \mt$.
Furthermore, the first two lines of List~B imply that if $(1)=\mt$, 
then there are at least three crossing-less quadrilaterals, which contradicts having seven crossings.
Therefore $(1) \ne \mt$.
So we update the above list to the following, which we call List~B1:

\begin{itemize}
\item $(1)\ne \mt$.
\item $(2)=\mt$ implies $(4)=\mt$ or $(8\mathrm{a})=\mt$.
\item  $(3)=\mt$ implies $(4)=\mt$.
\item $(4)=\mt$ implies $(2)=\mt$ or $(8\mathrm{b})=\mt$.
\item  $(5)=\mt$ implies $(9)=\mt$.
\item $(6\mathrm{a}) \ne \mt$; $(6\mathrm{b}) = \mt$.
\item  $(7)=\mt$ implies $(9)=\mt$ or $(8\mathrm{a})=\mt$.
\item $(9)=\mt$ implies $(7)=\mt$ or $(8\mathrm{b})=\mt$.
\end{itemize}

We have three cases, according to whether $(8)=\mt$, $(8\mathrm{a}) \ne \mt$, or $(8\mathrm{b}) \ne \mt$. 

\medskip

Case~1: $(8) = \mt$.
Since $\calP$ has seven crossings, 
there must be exactly one crossing-less quadrilateral other than the one in (8).
Hence, since $(3)=\mt$ implies $(4)=\mt$, we have $(3) \ne \mt$;
and since $(5)=\mt$ implies $(9)=\mt$, we have $(5) \ne \mt$. 
So exactly one of (2), (4), (7) or (9) is $\mt$.
This gives us the following four subcases.
 
\medskip

Case~1(a): $(2) =\mt$. 
This configuration is non-realizable by Lemma ~\ref{CorNR1}, via 236 and 145,
with the following justification: 
If $\ko$ has seven nontrivial links,
then each of 732, 736, and 754 is linked,
and edges 32, 36 and 54 satisfy condition~(i) of the lemma. 
If $\ko$ has only six nontrivial links,
then at least two of edges 32, 36 and 54 satisfy condition~(i); 
and edge 14 satisfies condition~(ii) 
since it is the only edge which crosses two disjoint edges 
(and hence 714 is the triangle component of the link with linking number~2).

\medskip

Case~1(b): $(4) =\mt$. 
Observe that $(1) \ne \mt$, $(3) \ne \mt$ and $(5) \ne \mt$  together imply that vertices  3 and 5 lie on one side of  $\C_{14}$, 2 and 6 on the other.
Let  $a= 14\cap 32$,  and $b= 14\cap 52$.
Since $16\cap 32 \ne \mt$ and $52\cap 16 = \mt$, applying Lemma \ref{4intersections} with $X=1$, $Y=4$, $A=2$, $B=6$, $C=5$, and $D=3$ implies that $b$ cannot be between $1$ and $a$ on edge $14$.
Hence  $a$ is between 1 and $b$.
This implies 3 and 4 are on different sides of $\C_{52}$, and therefore 32 cannot cross 54, which contradicts $(7) \ne \mt$.

\medskip

Case~1(c): $(7) =\mt$. 
This configuration is non-realizable by Lemma ~\ref{CorNR1}, via 325 and 416,
with the following justification: 
If $\ko$ has seven nontrivial links,
then each of 716, 732, and 752 is linked,
and edges 16, 32 and 52 satisfy condition~(i). 
If $\ko$ has only six nontrivial links,
then at least two of edges 16, 32 and 52 satisfy condition~(i); 
and edge 14 satisfies condition~(ii) (as in Case~1(a)).

\medskip 

Case~1(d): $(9)= \mt$. 
As in Case~1(b), $(1) \ne \mt$, $(3) \ne \mt$ and $(5) \ne \mt$  together imply that vertices  3 and 5 lie on one side of  $\C_{14}$, 2 and 6 on the other.
Let  $a= 14\cap 32$,  and $b= 14\cap 36$.
Since $54\cap 32 \ne \mt$ and $54\cap 36 = \mt$,  applying Lemma \ref{4intersections} 
with $X=4$, $Y=1$, $A=3$, $B=5$, $C=6$, and $D=2$ implies that 
$b$ is between 1 and $a$ on edge $14$.
This implies 1 and 6 are on the same side of $\C_{32}$, and hence 16 cannot cross 32, which contradicts $(2) \ne \mt$.

\medskip

Case~2: $(8\mathrm{a}) \ne \mt$.
Then we obtain the following implications from List~B1:
$(2)=\mt$ implies $(4)=\mt$; 
$(3)=\mt$ implies $(4)=\mt$; 
$(5)=\mt$ implies $(9)=\mt$; and  
$(7)=\mt$ implies $(9)=\mt$.
Since none of (1), (6), and (8) is $\mt$, and $\calP$ has seven crossings, 
there must be exactly two crossing-less quadrilaterals among  (2), (3), (4), (5), (7) and (9).
And the only pairs not ruled out by the above are:
(2) and (4); (3) and (4); (4) and (9); (5) and (9); (7) and (9).
Furthermore, if $(4) = \mt$ and $(9) = \mt$,
then we have $(1) \ne \mt$, $(2)\ne \mt$, $(3) \ne \mt$, and $(5) \ne \mt$
(since at most two quadrilaterals can be crossing-less);
hence the argument in Case~1(b) applies, giving a contradiction.
We're only left with four pairs:
(2) and (4); (3) and (4);  (5) and (9);  (7) and (9).

Now, if $\ko$ has seven nontrivial links,
then we can rule out all four pairs listed above 
since each of edges 16, 52, 36 and 54 has at least one crossing.
On the other hand, if $\ko$ has only six nontrivial links, we rule out the four pairs as follows.
If [$(2)= \mt $ and $(4)=\mt$] or [$(3)= \mt$ and $(4)=\mt$],
we get non-realizability by Lemma~\ref{CorNR1}, via 145 and 236, 
where edges 36 and 54 satisfy condition~(i) and one of edges 14 or 32 satisfies condition~(ii).
If [$(5)= \mt$ and $(9)=\mt$] or  [$(7) =\mt$ and $(9)=\mt$],
we get non-realizabilty by Lemma~\ref{CorNR1}, via 325 and 416, 
where edges 16 and 52 satisfy condition~(i) and one of the edges 14 or 32 satisfies condition~(ii).

\medskip

Case~3: $(8\mathrm{b}) \ne \mt$.
Then we obtain the following implications from List B1:  
$(3)=\mt$ implies $(4)=\mt$;  
$(4)=\mt$ implies $(2)=\mt$; 
$(5)=\mt$ implies $(9)=\mt$; and  
$(9)=\mt$ implies $(7)=\mt$.
Since none of (1), (6), and (8) is $\mt$, and $\calP$ has seven crossings, there must be a pair of crossing-less quadrilaterals among  (2), (3), (4), (5), (7) and (9).
And only three pairs are not ruled out by the above: (2) and (4); (2) and (7); (7) and (9).

We rule out the second pair, (2) and (7), as follows.
As we've seen before, $(1) \ne \mt$, $(3) \ne \mt$ and $(5) \ne \mt$  together imply that 
vertices  3 and 5 lie on one side of  $\C_{14}$, 2 and 6 on the other.
Let $a= 14\cap 32$ and  $b= 14\cap 52 $.
Since $16\cap 32 = \mt$ but $16\cap 52 \ne \mt$,
applying Lemma \ref{4intersections} with 
$X=1$, $Y=4$, $A=2$, $B=6$, $C=3$, and $D=5$
implies that $b$ is between 1 and $a$ on edge $14$.
Also, since $32\cap 54 = \mt$ but $36\cap 54 \ne \mt$,
letting $c =14\cap 36 $ and 
applying Lemma \ref{4intersections} with 
$X=4$, $Y=1$, $A=3$, $B=5$, $C=2$, and $D=6$
implies that $c$ is between 4 and $a$ on edge $14$.
It follows that edges 52 and 36 lie on different sides of $\C_{32}$ and therefore cannot cross each other,
contradicting $(8\mathrm{b}) \ne \mt$.

We rule out the first and third pairs,  [$(2) =\phi$ and $(4)=\phi$]  and [$(7) =\phi$ and $(9)=\phi$], as follows.
First suppose $K_{3,3,1}$ has seven nontrivial links.
Then edges 16 and 54 each have at least one crossing,
which rules out the first and third pairs, respectively.
Now suppose $\ko$ has only six nontrivial links.
If $(2)= \mt $ and $(4)=\mt$,
then this configuration is non-realizable by Lemma~\ref{CorNR1}, via 145 and 236, 
where edges 32 and 54 satisfy condition~(i) and one of the edges 14 or 36 satisfies condition~(ii).
If $(7)= \mt$ and $(9)=\mt$,
then this configuration is non-realizable by Lemma~\ref{CorNR1}, via 416 and 325, where edges 16 and 32 satisfy condition~(i) and one of edges 14 or 52 satisfies condition~(ii).

\medskip

\textbf{Part~2.} Suppose edge 56 has no crossings. 
Then we have $(8\mathrm{a})=\mt$ and $(6\mathrm{a})=\mt$;
and $\ko$ must have exactly six nontrivial links.
List~B becomes: 

\begin{itemize}
\item  $(1)=\mt$ implies $(2)=\mt$ or $(5)=\mt$.
\item  $(1)=\mt$ implies $(3)=\mt$ or $(7)=\mt$.
\item  $(4)=\mt$ implies $(2)=\mt$ or $(8)=\mt$.
\item  $(4)=\mt$ implies $(3)=\mt$ or $(6)=\mt$.
\item  $(9)=\mt$ implies $(5)=\mt$ or $(6)=\mt$.
\item  $(9)=\mt$ implies $(7)=\mt$ or $(8)=\mt$.
\end{itemize}

Since exactly two quadrilaterals are crossing-less, 
none of (1), (4) and (9) is $\mt$.
Now, given that $(1) \ne \mt$, if $(3)$ and $(5)$ were both $\ne \mt$, 
then vertices  3 and 5 would lie on one side of  $\C_{14}$, 2 and 6 on the other.
So 54 could not cross 16, and we'd have $(6) =\mt$.
Using a similar reasoning, we obtain the following 
(we're listing the pair (3), (5) again for ease of reference):
\begin{itemize}
\item
$(2),(6) \ne \mt$ implies $(5)=\mt$, since $(4) \ne  \mt$.
\item
$(2), (7)\ne \mt$ implies $(8)=\mt$, since $(1) \ne \mt$.
\item
 $(3), (5) \ne \mt $ implies $(6) = \mt$, since $(1) \ne \mt$.
\item
$(3), (8)\ne \mt$ implies $(7)=\mt$, since $(4) \ne \mt$.
\item
$(5), (8) \ne \mt$ implies $(2)=\mt$, since $(9) \ne \mt$.
\item
$(6), (7) \ne \mt$ implies $(3)=\mt$, since $(9) \ne \mt$.
\end{itemize}

Since none of (1), (4) and (9) is $\mt$
and exactly two quadrilaterals are crossing-less,
exactly two of (2), (3), (5), (6), (7) and (8) are $\mt$.
This gives fifteen possible pairs of crossing-less quadrilaterals.
Twelve of them are ruled out by the above list;
for example, if (2) and (5) are $\mt$, then,  for all $i \not \in \{2,5\}$, $(i) \ne \mt$; so
by the last line in the above list,
we get $(3) = \mt$, which is a contradiction.

The only remaining pairs are (2), (3); (5), (7); and (6), (8).
The first pair gives a non-realizable configuration by Lemma ~\ref{CorNR1}, 
via 145 and 236, where edges 14 and 32 satisfy condition~(i) and one of edges 36 or 54 satisfies condition~(ii).
The second pair gives a non-realizable configuration by Lemma ~\ref{CorNR1},
via 416 and 325, where edges 14 and 32 satisfy condition~(i) and one of edges 16 or 52 satisfies condition~(ii).
And the third pair gives a non-realizable configuration by Lemma ~\ref{CorNR1},
via 145 and 236, where edges 36 or 54 satisfy condition~(i) and one of edges 14 or 32 satisfies condition~(ii).

We conclude that a linear $\ko$ cannot have more than five nontrivial links.

\end{proof}

\bibliographystyle{amsplain}

\end{document}